\input ./style/arxiv-vmsta.cfg
\documentclass[numbers,compress]{vmsta}
\usepackage{enumerate}
\usepackage{dcolumn}

\volume{2}
\issue{2}
\pubyear{2015}
\firstpage{95}
\lastpage{106}
\doi{10.15559/15-VMSTA23}

\startlocaldefs

\newcommand{\rrvert}{\vert}
\newcommand{\llvert}{\vert}

\allowdisplaybreaks

\newtheorem{thm}{Theorem}
\newtheorem{prop}{Proposition}
\newtheorem{lem}{Lemma}
\newtheorem{cor}{Corollary}

\theoremstyle{definition}
\newtheorem{rem}{Remark}

\newcommand{\si}{\sigma}
\newcommand{\R}{\mathbb{R}}
\newcommand{\Ph}{\varPhi}
\newcommand{\de}{\delta}
\newcommand{\al}{\alpha}
\newcommand{\om}{\omega}

\endlocaldefs

\begin{document}
\begin{frontmatter}

\title{The rate of convergence to the normal law in terms of~pseudomoments}

\author[a]{\inits{Yu.}\fnm{\xch{Yuliya}{Yuliia}}\snm{Mishura}}\email{myus@univ.kiev.ua}
\author[a]{\inits{Ye.}\fnm{\xch{Yevheniya}{Yevheniia}}\snm{Munchak}\corref{cor1}}\email{yevheniamunchak@gmail.com}
\cortext[cor1]{Corresponding author.}
\address[a]{Taras Shevchenko National University of Kyiv, Volodymyrska
str. 64, 01601, Kyiv, Ukraine}
\author[b]{\inits{P.}\fnm{Petro}\snm{Slyusarchuk}}\email{petro\_slyusarchuk@ukr.net}
\address[b]{Uzhhorod National University, Pidhirna str. 46, 88000,
Uzhhorod, Ukraine}

\markboth{Yu. Mishura et al.}{The rate of convergence to the normal law in
terms of pseudomoments}

\begin{abstract}
We establish the rate of convergence of distributions of sums of
independent identically distributed random variables to the Gaussian
distribution in terms of truncated pseudomoments by implementing the
idea of Yu. Studnyev for getting estimates of the rate of convergence
of the order higher than $n^{-1/2}$.
\end{abstract}


\received{21 February 2015}
%
\revised{4 April 2015}
\accepted{10 April 2015}
\publishedonline{21 April 2015}
\end{frontmatter}

\section{Introduction}
Applications of the central limit theorem and other weak limit theorems
are closely connected to the rate of convergence to the limit
distribution. The rate of convergence was studied by many authors; see
\cite{petrov} and the references therein. The simplest result in this
direction is the Berry--Esseen inequality. Let $ \{\xi_n, n\ge
1 \}$ be a sequence of independent identically distributed random
variables (iidrvs) with distribution function $F (x)$, $E\xi_i=0$, and
$D\xi_i=\si^2<\infty$. Let~$\beta_3=\int_{\R}{\llvert x\rrvert ^{3}}dF(x)$
be the absolute 3rd moment, $\Ph_n(x) = P \{ \frac
{\xi
_1+\xi_2+\ldots+\xi_n}{\si\sqrt{n}}<x \} $, and let $\Ph
(x)$, $x\in
\mathbb{R,}$ be the standard normal distribution function. Then the
Berry--Esseen inequality states that
\[
\sup_{x\in\mathbb{R}} \bigl\llvert \Ph_n(x) - \Ph(x)\bigr\rrvert
\le \frac
{C\beta_3}{\si^3\sqrt{n}}.
\]
This estimate gives the rate of convergence $O(n^{-1/2})$ and the
asymptotic expansions of the distribution function of the sum of iidrvs
in terms of semiinvariants, presented in the book \cite{petrov}. The
same rate of convergence was obtained by Paulauskas \cite{Paulauskas}
in terms of pseudomoments. Let $\si=1$. Then the ``pseudomoment''
function is defined as $H(x)=F(x)-\Ph(x)$, the (absolute) third
pseudomoment is defined as
$\nu=\int_{\R}\llvert x\rrvert ^3\llvert dH(x)\rrvert $, and
we have
\[
\sup_{x\in\mathbb{R}} \bigl\llvert \Ph_n(x) - \Ph(x)\bigr\rrvert
\le C \max \bigl(\nu,\nu^{\frac{1}{4}} \bigr)n^{-\frac{1}{2}}.
\]
However, this rate of convergence is slow, for instance, from the point
of view of financial applications. The conditions that allow one to
improve the rate of convergence were formulated by several authors.
After the introduction of pseudomoments in \cite{Zolotarev1}, they are
widely used in limit theorems. Zolotarev~\cite{Zolotarev2} obtained
very general estimates in the central limit theorem using a different
type of pseudomoments. Studnyev \cite{Studnyev} obtained the following
estimate of the rate of convergence in terms of pseudomoments. Let
$ \{\xi_n, n\ge1 \}$ be centered iidrvs with unit variance and
characteristic function $f(t)$, $\mu_k=\int_{R} x^k dH(x) $ be the
$k$th-order pseudomoment, and $V(x)=V_{-\infty}^{x}H(z)$ be the
variation of the~function $H$.

\begin{prop}[\cite{Studnyev}]
Let $F(x)$ have finite moments up to the $q$th order for some $q\ge3$
and satisfy the Cramer condition
$\overline{\lim}_{\llvert t\rrvert \to\infty}\llvert f(t)\rrvert <1$. Then
\begin{align*}
\sup_{x\in\mathbb{R}} \bigl|\Ph_n(x) - \Ph(x)\bigr|
= O \Biggl(\sum_{k=3}^{q}
\frac{\llvert \mu_k \rrvert }{n^\frac{k-2}{2}}+ \frac{1}{n^\frac{q-2}{2}}\frac{1}{\sqrt{n}}
\int_{0}^{\sqrt{n}}dx\int_{\llvert z\rrvert >x}
\llvert z\rrvert ^q dV(z) \Biggr).
\end{align*}
\end{prop}
We can see that the condition $\mu_k=0$, $3\leq k\leq r $ supplies the
rate of convergence $O(n^{\frac{-r+1}{2}})$.
The rate of convergence was also studied in \cite{Bojaryshcheva,Honak,Sl-Pol}. In our work, we use a different type of
pseudomoments and
get the same rate of convergence avoiding the Cramer condition.
Instead, we impose the boundedness of the truncated pseudomoments and
integrability of the characteristic function.

\section{Generalization of Studnyev's estimate. Main results}

Let, as before, $ \{\xi_n, n\ge1 \}$ be a sequence of iidrvs
with $E\xi_i=0$, $D\xi_i=\si^2\in(0, \infty)$, distribution
function $F
(x)$, and characteristic function $f (t)$, and let $\Ph_n(x)$, $x\in
\mathbb{R}$, be the distribution function of the random variable
\[
S_n= (\si\sqrt{n} )^{-1} (\xi_1+
\xi_2+\cdots+\xi _n ).
\]

We assume that, for some $m\ge3$, there exist the pseudomoments
\[
\mu_k=\int_{\R}{x}^{k} {d} {H(x)},
\quad k=3,\ldots,m\in\mathbb{N},
\]
where $H(x)=F(x\si)-\Ph(x)$.
The truncated pseudomoments are defined as

\[
\nu_{n}^{(1)}(m)=\int_{|x|\le\si\sqrt{n}}\llvert x
\rrvert ^{m+1}\bigl\llvert dH(x)\bigr\rrvert
\]
(``truncation from above'')
and
\[
\nu_{n}^{(2)}(m)=\int_{|x|>\si\sqrt{n}}\llvert x
\rrvert ^{m}\bigl\llvert dH(x)\bigr\rrvert
\]
(``truncation from below'').

\begin{thm}
\label{t1}
Let the following conditions hold\emph{:}
\begin{itemize}
\item[\rm(i)] The characteristic function is integrable\emph{:} $A=\int_{\mathbb{R}}|f(t)|dt<\infty$\emph{;}

\item[\rm(ii)] The pseudomoments up to order $m$ equal zero, and the
truncated pseudomoments are bounded\emph{:}
\begin{align*}
&\mu_k=0,\quad  k=3,\ldots,m, \ \text{for some} \ m\ge3,\quad  \text {and}\\
&\nu_{n}(m)=\max \bigl\{\nu_{n}^{(1)}(m),
\nu_{n}^{(2)}(m) \bigr\} <\frac{1}2e^{-\frac{3}2}.
\end{align*}
\end{itemize}
Then, for all $n\ge2$,
\begin{align*}
\sup_{x\in\mathbb{R}} \bigl|\Ph_n(x) \,{-}\, \Ph(x)\bigr|
\le2C^{(1)}_m\frac{\nu_{n}^{(1)}(m)}{n^{\frac{m-1}{2}}}\,{+}\,
2C^{(2)}_m\frac{\nu_{n}^{(2)}(m)}{n^{\frac{m-2}{2}}} \,{+}\,
\frac{\si A}{\pi}b^{n-1}\,{+}\,
\nu_{n}(m)\frac{4e^{\frac{3}{2}}}{\pi}\frac{e^{-\frac{n}{2}}}{n},
\end{align*}
where
\begin{align*}
&C^{(1)}_m=\frac{12^{\frac{m+1}{2}}\varGamma(\frac{m+1}{2})}{\pi(m+1)!},\qquad
C^{(2)}_m=2C^{(1)}_{m-1},\\
&b=\exp \biggl\{-\frac{\pi^2}{24A^{2}\si^{2} (2+\pi)^{2}} \biggr\}<1.
\end{align*}
\end{thm}

\begin{cor}
\label{t1'}
Let $\xi_i$ be a random variable with bounded density $p(x)\le A_1$.
Suppose that condition \emph{(ii)} of Theorem \emph{\ref{t1}} holds.
Then, for all $n\ge3$,
\begin{align*}
\sup_{x\in\mathbb{R}} \bigl|\Ph_n(x) \,{-}\, \Ph(x)\bigr|
\,{\le}\,2C^{(1)}_m\frac{\nu_{n}^{(1)}(m)}{n^{\frac{m-1}{2}}}\,{+}\,
2C^{(2)}_m \frac{\nu_{n}^{(2)}(m)}{n^{\frac{m-2}{2}}} \,{+}\,
2\si A_1 b_1^{n-2}\,{+}\,\nu_{n}(m)\frac{4e^{\frac{3}{2}}}{\pi}
\frac{e^{-\frac{n}{2}}}{n},
\end{align*}
where $b_1=\exp \{-\frac{1}{96A_1^2\si^2 (2+\pi)^2}\}<1$.
\end{cor}

Note assumption (i) implies the existence of the density $p_n (x)$ of
the random variable $S_n$. Also, let $\phi(x)$ be the density of the
standard normal law.
\begin{thm}
\label{t2}
Let the conditions of Theorem \emph{\ref{t1}} hold.
Then, for all $n\ge2$,
\begin{align*}
\sup_{x\in\mathbb{R}} \bigl|p_n(x) - \phi(x)\bigr|
\le C_{m}^{(3)}\frac{\nu_{n}^{(1)}(m)}{n^{\frac
{m-1}{2}}}+C_{m}^{(4)}
\frac{\nu_{n}^{(2)}(m)}{n^{\frac{m-2}{2}}} +b^{n-1}\frac{\si\sqrt{n}}{2\pi}A+\nu_{n}(m)
\frac{e^{\frac
{3}{2}}}{\pi
}\frac{e^{-\frac{n}{2}}}{n},
\end{align*}
where
\[
C_{m}^{(3)}=\frac{12^{\frac{m+2}{2}}\varGamma (\frac
{m}{2}+1
)}{4\pi(m+1)!}, \qquad  C^{(4)}_m=2C^{(3)}_{m-1}.
\]
\end{thm}
\section{Auxiliary results. Proofs of the main results}
First we prove two auxiliary results. Denote $\om(t)=\llvert f
(\frac
{t}{\si} )-e^{-\frac{t^2}{2}}\rrvert $.
\begin{lem}
\label{l1}
Let $\mu_k=0$, $k=3,\ldots,m$. Then, for all $t\in\R$,
\[
\om(t)\le\frac{\llvert t\rrvert ^{m+1}}{(m+1)!}\nu _{n}^{(1)}(m)+
\frac
{2|t|^m}{m!}\nu_{n}^{(2)}(m).
\]
\end{lem}
\begin{proof}
Recall that
$f(t)=\int_{-\infty}^{\infty}{e^{itx}dF(x)}$. Therefore,
$f (\frac{t}{\si} )=\int_{-\infty}^{\infty
}{e^{\frac
{itx}{\si}}dF(x)}=\int_{-\infty}
^{\infty}{e^{itx}dF(x\si)}$.
By the condition of the lemma, the pseudomoments up to order $m$ equal
zero. Hence, it is easy to deduce that
\begin{align*}
\om(t)&=\biggl\llvert \int_{\R}{e^{itx}dF(x
\si)}-\int_{\R
}{e^{itx}d\Ph(x)}\biggr\rrvert =\biggl\llvert
\int_{\R}{e^{itx}d \bigl(F(x\si )-\Ph(x) \bigr)}\biggr
\rrvert
\\
&=\Biggl\llvert \int_{\R}{ \Biggl(e^{itx}-\sum
_{k=0}^{m}\frac
{(itx)^k}{k!} \Biggr)dH(x)}\Biggr
\rrvert \le\int_{\R}{\Biggl\llvert e^{itx}-\sum
_{k=0}^{m}\frac{(itx)^k}{k!}\Biggr\rrvert \bigl
\llvert dH(x)\bigr\rrvert }.
\end{align*}
Using the inequality (\cite{Zolotarev}, p. 372)
\[
\biggl\llvert e^{i\al}-1-\cdots-\frac{(i\al)^m}{m!}\biggr\rrvert \le
\frac
{2^{1-\delta
}|\al|^{m+\delta}}{m!(m+1)^\delta},\quad m=0,1,\ldots,\ \de\in[0,1],
\]
with $\de= 1$ and $\de= 0$ we obtain
\begin{align*}
\om(t)&\le\int_{|x|\le\si\sqrt{n}}{\Biggl\llvert
e^{itx}-\sum_{k=0}^{m}
\frac{(itx)^k}{k!}\Biggr\rrvert \bigl\llvert dH(x)\bigr\rrvert }+\int_{|x|>\si\sqrt{n}}{
\Biggl\llvert e^{itx}-\sum_{k=0}^{m}
\frac{(itx)^k}{k!}\Biggr\rrvert \bigl\llvert dH(x)\bigr\rrvert }\\
&\le\int_{|x|\le\si\sqrt{n}}\frac
{|tx|^{m+1}}{(m+1)!}\bigl\llvert dH(x)\bigr\rrvert \quad+
\int_{|x|>\si\sqrt{n}}{\frac
{2|tx|^m}{m!}\bigl\llvert dH(x)\bigr\rrvert }\\
&=\frac{|t|^{m+1}}{(m+1)!}\nu_{n}^{(1)}(m)+\frac{2|t|^m}{m!}\nu_{n}^{(2)}(m).
\end{align*}
The lemma is proved.
\end{proof}
Now, denote $T_{1}(n,m)=\sqrt{-2\ln(2e\nu_{n}(m))}$. Then, in turn, we
have that $\nu_{n}(m)=\frac{1}{2e}\exp\{-\frac{1}2T^2_{1}(n,m)\}$. Note
also that condition (ii) implies \hbox{$T_{1}(n,m)\geq1$}.
\begin{lem}
\label{l2}
Suppose that condition \emph{(ii)} of Theorem \emph{\ref{t1}} holds.
\begin{enumerate}[\rm1)]
\item For $|t|\le T_{1}(n,m)$, the characteristic function allows the
following bound\emph{:}\break
$\llvert f (\frac{t}{\si} )\rrvert \le e^{-\frac{t^2}{6}}$.
\item For $|t|>T_{1}(n,m)$, the characteristic function allows the
following bound\emph{:}\break
$\llvert f (\frac{t}{\si} )\rrvert \le (2e+\frac{3}8 )\nu
_{n}(m)|t|^{m+1}$.
\end{enumerate}
\end{lem}
\begin{proof} Evidently,
$\llvert f (\frac{t}{\si} )\rrvert =\llvert f (\frac
{t}{\si
} )-e^{-\frac{t^2}{2}}+e^{-\frac{t^2}{2}}\rrvert \le e^{-\frac
{t^2}{2}}+\om(t)$. Now consider two cases.

$1)$ Let $| t |\le T_{1}(n,m)$. Then we can deduce from Lemma \ref{l1} that
\begin{align}\label{eq1}
\biggl\llvert f \biggl(\frac{t}{\si} \biggr)\biggr\rrvert
&\le e^{-\frac{t^2}{4}}\bigl(e^{-\frac{t^2}{4}}+e^{\frac{t^2}{4}}\om(t) \bigr)\nonumber\\
&\le e^{-\frac{t^2}{4}}\biggl(1+e^{\frac{t^2}{4}}\biggl(\frac{\llvert t\rrvert^{m+1}}{(m+1)!}
\nu_{n}^{(1)}(m)+\frac{2|t|^m}{m!}\nu_{n}^{(2)}(m)\biggr) \biggr)\nonumber\\
&\le e^{-\frac{t^2}{4}} \biggl(1+e^{\frac{T_{1}^{2}(n,m)}{4}}t^2
\biggl(\frac{T_{1}^{m-1}(n,m)}{(m+1)!}\nu_{n}^{(1)} (m)+\frac{2T_{1}^{m-2}(n,m)}{m!}
\nu_{n}^{(2)}(m) \biggr) \biggr)\nonumber\\
&\le e^{-\frac{t^2}{4}} \biggl(1+t^2e^{\frac{T_{1}^{2}(n,m)}{4}}{\nu_{n}(m)}
\biggl(\frac{T_{1}^{m-1}(n,m)}{(m+1)!} +\frac{2T_1^{m-2}(n,m)}{m!} \biggr) \biggr)\nonumber\\
&= e^{-\frac{t^2}{4}} \biggl(1+\frac{1}{2e}t^2e^{-\frac{T_{1}^{2}(n,m)}{4}}
\biggl(\frac{T_{1}^{m-1}(n,m)}{(m+1)!} +\frac{2T_1^{m-2}(n,m)}{m!} \biggr) \biggr).
\end{align}
Consider the function $f_1(x)=\exp\{-\frac{x^2}{4}\}x^{m-1}$. It
attains its maximal value at the point $x=\sqrt{2(m-1)}$, and this
value equals
\[
f_{1, \max}=\exp \biggl\{-\frac{m-1}{2} \biggr\} \bigl(2(m-1)
\bigr)^{\frac
{m-1}{2}}.
\]
Furthermore,
\begin{align*}
\exp \biggl\{-\frac{m-1}{2} \biggr\}\frac{(2(m-1))^{\frac{m-1}{2}}}{(m+1)!}
&\leq\frac{\exp \{-\frac{m-1}{2} \}(2(m-1))^{\frac{m-1}{2}}}{m(m+1)
\sqrt{2\pi(m-1)}(m-1)^{m-1}e^{-(m-1)}}\\
&= \biggl(\frac{2e}{m-1} \biggr)^{\frac{m-1}{2}}\frac{1}{\sqrt{2\pi(m-1)}}
\frac{1}{m(m+1)}\\
&\le\frac{1}{m(m+1)}.
\end{align*}
The last fraction attains its maximal value at the point $m=3$. Therefore,
\[
\exp \biggl\{-\frac{T_{1}^{2}(n,m)}{4} \biggr\}\frac
{T_{1}^{m-1}(n,m)}{(m+1)!}\leq
\frac{1}{12}.
\]
Similarly,
\[
\exp \biggl\{-\frac{T_{1}^{2}(n,m)}{4} \biggr\}\frac
{2T_1^{m-2}(n,m)}{m!}\le
\frac{1}{3}.
\]
From \eqref{eq1} together with two last bounds it follows that
\begin{align*}
\biggl\llvert f \biggl(\frac{t}{\si} \biggr)\biggr\rrvert
&\le e^{-\frac{t^2}{4}} \biggl(1+\frac{1}{2e}t^2
\biggl(\frac{1}{12}+\frac{1}{3} \biggr) \biggr) \leq e^{-\frac{t^2}{4}}
\biggl(1+\frac{1}{12}t^2 \biggr)\\
&\leq e^{-\frac{t^2}{4}}e^{\frac{t^2}{12}}\leq e^{-\frac{t^2}{6}},
\end{align*}
and the proof of the first statement follows.

$2)$ Now, let $|t|>T_1(n,m)$. Then we get from Lemma \ref{l1} that
\begin{align}\label{eq2}
\biggl\llvert f \biggl(\frac{t}{\si} \biggr)\biggr\rrvert
&\le e^{-\frac{t^2}{2}}+\om(t)\nonumber\\
&\le e^{-\frac{T^2_1(n,m)}{2}}+\frac{\llvert t\rrvert ^{m+1}}{(m+1)!}\nu _{n}^{(1)}(m)
+\frac{2|t|^m}{m!}\nu_{n}^{(2)}(m)\nonumber\\
&\le\nu_{n}(m) \biggl(2e+\frac{\llvert t\rrvert ^{m+1}}{(m+1)!}
+\frac{2|t|^m}{m!}\biggr).
\end{align}
Recall that $T_{1}(n,m)>1$. Then $|t|>T_1(n,m)>1$, and from \eqref
{eq2} we get that
\begin{equation*}
\biggl\llvert f \biggl(\frac{t}{\si} \biggr)\biggr\rrvert \le
\nu_{n}(m) \biggl(2e|t|^{m+1}+\frac{|t|^{m+1}}{24}+
\frac
{2|t|^{m+1}}{6} \biggr) \le \biggl(2e+\frac{3}8 \biggr)
\nu_{n}(m)|t|^{m+1},
\end{equation*}
whence the proof follows.
\end{proof}

Now we are in position to prove the main results.

\begin{proof}[Proof of Theorem \ref{t1}]
Let $F$ and $G$ be two distribution functions with characteristic
functions $f$ and $g$, respectively, and suppose that $G$ has a density
function, which we denote $G'$. We shall use the following inequality
from \cite{Loeve}, p.~297:
\begin{equation*}
\sup_{x\in\mathbb{R}}\bigl\llvert F(x)-G(x)\bigr\rrvert \le\frac
{2}{\pi}
\int_{0}^{T}\bigl|f(t)-g(t)\bigr|\frac{dt}{t}+
\frac{24\sup
|G^{'}|}{\pi T}.
\end{equation*}
Taking $F (x) = \varPhi_n(x)$ and $G (x) = \varPhi(x)$, we have
\begin{equation}\label{2}
\rho_n:=\sup_{x\in\mathbb{R}}
\bigl\llvert \varPhi_n(x)-\varPhi (x)\bigr\rrvert \le
\frac{2}{
\pi}\int_{0}^{T}\biggl\llvert f^n
\biggl(\frac{t}{\si\sqrt
{n}} \biggr)-e^{-\frac{t^2}{2}}\biggr\rrvert
\frac{dt}{t}+\frac{24}{\pi\sqrt{2\pi} T}.
\end{equation}
Let $n\ge2$.
First, from the elementary inequality
\[
\bigl\llvert u^n-v^n\bigr\rrvert \le\llvert u-v\rrvert
\sum_{k=1}^{n}|u|^{k-1}|v|^{n-k}
\]
and from Lemma \ref{l1} it follows that, for $t\le T_1(n,m)\sqrt{n}$,
\begin{align}\label{3}
\biggl\llvert f^n \biggl(\frac{t}{\si\sqrt{n}} \biggr) -e^{-\frac{t^2}{2}}\biggr\rrvert
&\le \om \biggl(\frac{t}{\sqrt{n}} \biggr)
\sum_{k=1}^{n}\biggl\llvert f\biggl(\frac{t}{\si\sqrt{n}} \biggr) \biggr\rrvert ^{k-1}
e^{-\frac{t^2}{2n}(n-k)}\nonumber\\
&\le\om \biggl(\frac{t}{\sqrt{n}} \biggr)\sum_{k=1}^{n}
e^{-\frac{t^2}{6}\frac{n-1}{n}}
\le\om \biggl(\frac{t}{\sqrt{n}} \biggr)ne^{-\frac{t^2}{12}}\nonumber\\
&\le n \biggl(\frac{|t|^{m+1}}{(m+1)!n^{\frac{m+1}{2}}}\nu_{n}^{(1)} (m)
+\frac{2|t|^m}{m!n^\frac{m}{2}}\nu_{n}^{(2)}(m) \biggr)
\exp \biggl\{ -\frac{t^2}{12} \biggr\}\nonumber\\
&=\exp \biggl\{-\frac{t^2}{12} \biggr\} \biggl(\frac{|t|^{m+1}}{(m+1)!
n^{\frac{m-1}{2}}}\nu_{n}^{(1)} (m)+\frac{2|t|^{m}}{m!n^\frac{m-2}{2}}
\nu_{n}^{(2)}(m)\biggr).
\end{align}
Second, introduce the following notation:
\begin{align*}
C^{(1)}_{m,n}=\frac{12^{\frac{m-1}{2}}\varGamma(\frac
{m+1}{2})}{2n^{\frac
{m-1}{2}}(m+1)!}, \qquad  C^{(2)}_{m,n}=2C^{(1)}_{m-1,n},\\
T_2(n,m)=\frac{1}{\sqrt{2\pi} (C^{(1)}_{m,n}\nu
_{n}^{(1)}(m)+C^{(2)}_{m,n}\nu_{n}^{(2)}(m) )}.
\end{align*}
Then
\begin{equation}\label{equ5}
\frac{24}{\pi\sqrt{2\pi} T_2(n,m)}=C^{(1)}_m
\frac{\nu_{n}^{(1)}(m)}{n^{\frac{m-1}{2}}}+C^{(2)}_m
\frac{\nu_{n}^{(2)}(m)}{n^{\frac{m-2}{2}}}.
\end{equation}

Let $T_3 (n,m)= (T_1 (n,m)\sqrt{n})\wedge T_2(n,m)$. Then it follows
from (\ref{2}) and \eqref{equ5} that
\begin{align}\label{4} %
\rho_n &\le\frac{2}{\pi}\int_{0}^{T_3(n,m)}\biggl\llvert f^n
\biggl(\frac{t}{\si\sqrt{n}} \biggr)-e^{-\frac{t^2}{2}}\biggr\rrvert \frac{dt}{t}
+\frac{2}{\pi}\int_{T_3(n,m)}^{T_2(n,m)}\biggl\llvert f
\biggl(\frac{t}{\si\sqrt{n}} \biggr)\biggr\rrvert ^n\frac{dt}{t}\nonumber\\
&\quad+\frac{2}{\pi}\int_{T_3(n,m)}^{T_2(n,m)}e^{-\frac{t^2}{2}}
\frac{dt}{t} +\frac{24}{\pi\sqrt{2\pi} T_2(n,m)}\nonumber\\
&=I_1(n,m)+I_2(n,m)+I_3(n,m)+C^{(1)}_m\frac{\nu_{n}^{(1)}(m)}{n^{\frac{m-1}{2}}}
+C^{(2)}_m\frac{\nu_{n}^{(2)}(m)}{n^{\frac{m-2}{2}}}.
\end{align}
Since $T_3(n,m)\le T_1(n,m) \sqrt{n}$, from (\ref{3}) we get that
\begin{align}\label{5}
I_1(n,m)&=\frac{2}{\pi}\int_{0}^{T_3(n,m)}\biggl\llvert f^n
\biggl(\frac{t}{\si\sqrt{n}} \biggr)-e^{-\frac{t^2}{2}}\biggr\rrvert \frac{dt}{t}\nonumber\\
&\le\frac{2}{\pi}\int_{0}^{T_3(n,m)} \biggl(\frac{t^{m}}{(m+1)!
n^{\frac{m-1}{2}}}\nu_{n}^{(1)}(m)+\frac{2t^{m-1}}{m!n^{\frac{m-2}{2}}}\nu_{n}^{(2)}(m)\biggr)
e^{-\frac{t^2}{12}}dt\nonumber\\
&\le\frac{12^{\frac{m+1}{2}}\varGamma(\frac{m+1}{2})} {\pi(m+1)!n^{\frac{m-1}{2}}}
\nu_{n}^{(1)}(m)+\frac{2\cdot12^{\frac{m}{2}}
\varGamma(\frac{m}{2})} {\pi m!n^{\frac{m-2}{2}}}{\nu_{n}^{(2)}(m)}\nonumber\\
&=C^{(1)}_m\frac{\nu_{n}^{(1)}(m)}{n^{\frac{m-1}{2}}}+C^{(2)}_m
\frac{\nu_{n}^{(2)}(m)}{n^{\frac{m-2}{2}}}.
\end{align}
If $T_{3}(n,m)=T_{2}(n,m)$, then $I_2(n,m)=0$ and $I_3(n,m)=0$.
Therefore, we consider the case $T_{3}(n,m)=T_{1}(n,m)\sqrt{n}$. Then
\begin{equation*}
I_2(n,m)=\frac{2}{\pi}\int_{T_3(n,m)}^{T_2(n,m)}\biggl\llvert f
\biggl(\frac{t}{\si\sqrt{n}} \biggr) \biggr\rrvert ^n\frac{dt}{t}=
\frac{2}{\pi}\int_{T_1(n,m)/\sigma
}^{T_2(n,m)/\sigma
\sqrt{n}}\bigl\llvert f ({t} )\bigr\rrvert
^{n}\frac{dt}{t}.
\end{equation*}
Now we apply the following result of Statulevi\v cius \cite
{Statulyavichus}: if a random variable with characteristic function $f
(t)$ has a density $p (x) \le d <\infty$ and variance $\si^ 2$, then,
for any $t\in\mathbb{R}$,
\begin{equation}
\label{equation88}\bigl\llvert f(t)\bigr\rrvert \le\exp \biggl\{ -\frac
{t^2}{96d^{2}(2\si|t|+\pi)^2}
\biggr\}.
\end{equation}
It follows from condition (i) that the density $p(x)$ of any $\xi_n$
can be obtained as the inverse Fourier transform $p(x)=\frac{1}{2\pi
}\int_{\R}e^{-itx}f(t)dt$ and $p(x)\le\frac{1}{2\pi}\int_{-\infty
}^{\infty}|f(t)|dt=\frac{A}{2\pi}$.\vadjust{\goodbreak}
Besides, the function $\frac{t^2}{(2\si t+\pi)^2}$ is increasing for
$t> 0$. Therefore, for $| t |\ge T_1(n,m)/\sigma$ (recall that $T_1(n,m)>1$),
\[
\bigl|f(t)\bigr|\le\exp \biggl\{-\frac{\pi^2}{24A^2\si^2 (2+\pi
)^2} \biggr\}=:b,
\]
and $0<b<1$.
Then
\begin{equation}
\label{7} I_2(n,m)=\frac{2}{\pi}\int_{T_1(n,m)/\sigma}^{T_2(n,m)/\sigma\sqrt{n}}
\bigl\llvert f ({t} )\bigr\rrvert ^{n}\frac{dt}{t} \le
\frac{2\si}{\pi}b^{n-1}\int_{0}^{\infty}\bigl|f(t)\bigr|{dt}
=\frac{\si
A}{\pi
}b^{n-1}.
\end{equation}
Finally, we bound $I_3(n,m)$. Note that $I_3(n,m)$ is nonzero only if
$T_1(n,m)\sqrt{n}<T_2(n,m)$. Therefore,
\begin{align} \label{9}
I_3(n,m)&\le\frac{2}{\pi} \int_{T_1(n,m)\sqrt{n}}^{\infty}e^{-\frac{t^2}{2}}
\frac{dt}{t}\le\frac{2}{\pi}\frac{e^{-\frac{nT^2_1(n,m)}{2}}}{nT^2_1(n,m)}\nonumber\\
&\le \frac{2(2e\nu_n(m))^n}{\pi n}\le\frac{4e\nu_n(m)}{\pi n}\bigl(2e\nu _n(m)\bigr)^{n-1}
\le\nu_n(m)\frac{4e\cdot e^{{-\frac{n-1}{2}}}}{\pi n}\nonumber\\
&=\nu_{n}(m)\frac{4e^\frac{3}{2}}{\pi}\frac{e^{-\frac{n}{2}}}{n}.
\end{align}
Relations (\ref{4})--(\ref{9}) supply the proof of Theorem \ref{t1}.
\end{proof}

\begin{rem}\label{rem1}
Let the following conditions hold: $\mu_k=0$, $k=3,\ldots,m$, $m\ge3$.
Then
\begin{align*}
\sup_{x\in\mathbb{R}} \bigl|\Ph_1(x) - \Ph(x)\bigr|
&=\sup_{x\in\mathbb{R}}\bigl|F(x\si) - \Ph(x)\bigr|\\
&\le \biggl(\frac{6}{\pi(m+1)!}+\frac{2}{\pi\sqrt{2\pi}} \biggr)\max \bigl(
\nu_{1}(m), \bigl(\nu_{1}(m) \bigr)^{\frac{1}{m+2}} \bigr).
\end{align*}

Indeed, let $n=1$. Theorem is obvious when $\nu_{1}(m)>1$. Let $\nu
_{1}(m)\le1$. Put $T=(\nu_1(m))^{-\frac{1}{m+2}}$ into \emph{(\ref
{2})}. Then from Lemma \ref{l1} it follows that
\begin{align*}
\rho_1&=\sup_{x\in\mathbb{R}} \bigl|\Ph_1(x) - \Ph(x)\bigr|
=\sup_{x\in\mathbb{R}} \bigl|F(x\si) - \Ph(x)\bigr|\\
&\le\frac{2}{\pi}\int_{0}^{T}
\biggl(\frac{|t|^{m+1}}{(m+1)!}\nu_{1}^{(1)} (m)
+\frac{2|t|^m}{m!}\nu_{1}^{(2)}(m)\biggr)\frac{dt}{t}
+\frac{24}{\pi\sqrt{2\pi}T}\\
&\le\frac{2}{\pi} \biggl(\frac{T^{m+1}}{(m+1)\cdot(m+1)!}\nu _{1}^{(1)}(m)
+\frac{2T^m}{m\cdot m!}\nu_{1}^{(2)}(m) \biggr)
+\frac{24}{\pi\sqrt{2\pi}}\bigl(\nu_1(m)\bigr)^\frac{1} {m+2}\\
&\le\bigl(\nu_1(m)\bigr)^\frac{1} {m+2} \biggl(\frac{2}{\pi}\frac{3}{(m+1)!}
+\frac{24}{\pi\sqrt{2\pi}} \biggr).
\end{align*}
\end{rem}

\begin{proof}[Proof of Corollary \ref{t1'}]
\xch{Proof is}{is} similar to that of Theorem \ref{t1}.
We apply inequality \eqref{equation88} and recall again that
the function $\frac{t^2}{(2\si t+\pi)^2}$ is increasing for $t> 0$.
Therefore, for $| t |\ge T_1(n,m)/\sigma$ (recall that $T_1(n,m)>1$),
\[
\bigl|f(t)\bigr|\le\exp \biggl\{-\frac{1}{96A_1^2\si^2 (2+\pi
)^2} \biggr\}=:b_1,
\]
and $0<b_1<1$. It follows from \cite{Feller}, p. 510, that $\int
_{-\infty}^{\infty}|f(t)|^2dt\le2\pi A_1 $.
Therefore,
\begin{equation*}
I_2(n,m)=\frac{2}{\pi}\int_{T_1(n,m)/\sigma
}^{T_2(n,m)/\sigma
\sqrt{n}}\bigl\llvert f
({t} )\bigr\rrvert ^{n}\frac{dt}{t} \le\frac{2\si}{\pi}b_1^{n-2}
\int_{0}^{\infty}\bigl|f(t)\bigr|^2{dt} =2\si
A_1b_{1}^{n-2}.
\end{equation*}
Corollary \ref{t1'} is proved.
\end{proof}
\begin{rem} For $n=2$, we can get estimates similar to those in Remark
\ref{rem1}.
\end{rem}

\begin{proof}[Proof of Theorem \ref{t2}]
As it was mentioned before, condition (i) implies the existence of a
density for the random variable $\xi_k$, so the random variable~$S_n$
has the density
\[
p_n(x)=\int_{-\infty}^{\infty}e^{-itx}f^n
\biggl(\frac
{t}{\si\sqrt
{n}} \biggr)dt.
\]
Since $\phi(x)=\frac{1}{\sqrt{2\pi}}e^{-\frac{x^2}{2}}$ is the density
of the standard normal law, we have $\phi(x)=\frac{1}{2\pi}\int
_{-\infty}^{\infty}e^{-itx}e^{-\frac{t^2}{2}}dt$ and
\begin{align*}
\bigl\llvert p_n(x)-\phi(x)\bigr\rrvert &=\frac{1}{2\pi}\biggl
\llvert \int_{-\infty
}^{\infty}e^{-itx}f^n \biggl(
\frac{t}{\si\sqrt{n}} \biggr)dt-\int _{-\infty}^{\infty}e^{-itx}e^{-\frac{t^2}{2}}dt
\biggr\rrvert
\\
&\le\frac{1}{2\pi}\int_{-\infty}^{\infty}\biggl\llvert f^n
\biggl(\frac
{t}{\si\sqrt{n}} \biggr)-e^{-\frac{t^2}{2}}\biggr\rrvert dt.
\end{align*}
Therefore,
\begin{align}\label{11} %
\bigl\llvert p_n(x)-\phi(x)
\bigr\rrvert &\le\frac{1}{2\pi}\int_{|t|\le
T_{1}(n,m)\sqrt{n}}\biggl\llvert f^n \biggl(
\frac{t}{\si\sqrt{n}} \biggr)-e^{-\frac
{t^2}{2}}\biggr\rrvert dt\nonumber\\
&\quad+\frac{1}{2\pi}\int
_{|t|>T_{1}(n,m)\sqrt{n}}\biggl\llvert f \biggl(
\frac{t}{\si\sqrt{n}} \biggr)\biggr\rrvert ^n dt\nonumber\\
&\quad+\frac{1}{2\pi}\int
_{|t|>T_{1}(n,m)\sqrt{n}}e^{-\frac{t^2}{2}}dt\nonumber\\
&=I_1+I_2+I_3.
\end{align}
From the conditions of the theorem, Lemmas \ref{l1} and \ref{l2}, and
from (\ref{3}) $(n\ge2)$ we obtain the following: for $|t|\le
T_1(n,m)\sqrt{n}$ and $\nu_{n}(m)<\frac{1}{2}e^{-\frac{3}{2}}$,
\begin{align}
\label{12} %
I_1&=\frac{1}{2\pi}\int
_{|t|\le T_{1}(n,m)\sqrt{n}}
\biggl\llvert f^n \biggl(\frac{t}{\si\sqrt{n}} \biggr)-e^{-\frac{t^2}{2}}
\biggr\rrvert dt\nonumber\\
&\le\frac{1}{2\pi}\int
_{|t|\le T_{1}(n,m)\sqrt{n}} \biggl(\frac
{|t|^{m+1}\nu_{n}^{(1)}(m)}{(m+1)!n^{\frac{m-1}{2}}}+\frac
{2|t|^{m}\nu
_{n}^{(2)}(m)}{m!n^{\frac{m}{2}-1}}
\biggr)e^{-\frac{t^2}{12}}dt\nonumber\\
&\le\frac{12^{\frac{m+2}{2}}\varGamma (\frac{m}{2}+1)}{4\pi
(m+1)!}\frac{\nu_{n}^{(1)}(m)}{n^{\frac{m-1}{2}}}+\frac{2\cdot
12^{\frac{m+1}{2}}\varGamma (\frac{m-1}{2} )}{4\pi m!}\frac{\nu
_{n}^{(2)}(m)}{n^{\frac{m-2}{2}}}\nonumber\\
&=C_{m}^{(3)}\frac{\nu_{n}^{(1)}(m)}{n^{\frac
{m-1}{2}}}+C_{m}^{(4)} \frac{\nu_{n}^{(2)}(m)}{n^{\frac{m-2}{2}}}.
\end{align}
From the conditions of the theorem, similarly to (\ref{7}), we get
\begin{align}\label{13} %
I_2&=\frac{1}{2\pi}\int
_{|t|>T_{1}(n,m)\sqrt{n}}
\biggl\llvert f \biggl(\frac{t}{\si\sqrt{n}} \biggr)\biggr\rrvert ^n dt\nonumber\\
&=\frac{\si\sqrt{n}}{2\pi}\int
_{|z|>T_1(n,m)/\si}\bigl\llvert f(z)\bigr\rrvert ^{n}dz
\le b^{n-1}\frac{\si\sqrt{n}}{2\pi}A. \\
\label{14}
I_3&=\frac{1}{2\pi}\int_{|t|>T_{1}(n,m)\sqrt{n}}
e^{-\frac{t^2}{2}}dt\le\frac{1}{2\pi}\int_{T_1(n,m)\sqrt{n}}^{\infty}
e^{-\frac{t^2}{2}}dt\nonumber\\
&\le\frac{e^{-\frac{nT_{1}^{2}(n,m)}{2}}}{2\pi
\sqrt{n}T_{1}(n,m)}\le \frac{ (2e\nu_{n}(m) )^{n}}{2\pi\sqrt{n}}\nonumber\\
&\le\frac{e\nu_{n}(m)}{\pi\sqrt{n}}e^{-\frac{n-1}{2}}=\nu _{n}(m)
\frac{e^{\frac{3}{2}}}{\pi}\frac{e^{-\frac{n}{2}}}{\sqrt{n}}.
\end{align}
Relations (\ref{11})--(\ref{14}) supply the proof of Theorem \ref{t2}.
\end{proof}

\section{Example}
We give an example of application of Theorem \ref{t1'}.
It is similar to the example of \cite{Zolotarev}, p.\ 375, where the
discrete distribution was considered. Define the~distribution function
$F(x)$ as
\[
F(x)= %
\begin{cases}
\varPhi(x)&\text{if $|x|\ge\epsilon$,}\\
\varPhi(-\epsilon)&\text{if $-\epsilon<x<-\theta\epsilon$,}\\
\varPhi(\epsilon)&\text{if $\theta\epsilon<x<\epsilon$,}\\
\frac{1}{2}+\frac{\varPhi(\epsilon)-\frac{1}{2}}{\theta\epsilon
}x&\text{if
$|x|\le\theta\epsilon$,}
\end{cases} %
\]
where $0<\epsilon<1$, and $0<\theta<1$ is the root of the equation
\begin{equation}
\label{exp}\int
_{0}^{\epsilon}x^2d\varPhi (x)=
\frac
{ (\theta\epsilon )^2}{3} \biggl(\varPhi(\epsilon)-\frac
{1}{2} \biggr).
\end{equation}
This equation has a unique solution because $\int_{0}^{\epsilon
}x^2d\varPhi(x)\le\frac{\epsilon^2}{3} (\varPhi(\epsilon)-\frac
{1}{2} )$.
Indeed, on one hand,
\[
\varphi(x)=\frac{1}{\sqrt{2\pi}} \biggl(1-\frac{x^2}{2}+\frac
{x^4}{2^{2}2!}-
\frac{x^6}{2^{3}3!}+\frac{x^8}{2^{4}4!}-\cdots \biggr),
\]
and therefore,
\[
\varPhi(\epsilon)-\frac{1}{2}=\int
_{0}^{\epsilon}\varphi (x)dx=
\frac
{1}{\sqrt{2\pi}} \biggl(\epsilon-\frac{\epsilon^3}{6}+\frac
{\epsilon
^5}{40}-
\frac{\epsilon^7}{7\cdot2^{3}3!}+\cdots \biggr).
\]
So
\[
\varPhi(\epsilon)-\frac{1}{2}\ge\frac{1}{\sqrt{2\pi}} \biggl(\epsilon-
\frac
{\epsilon^3}{6} \biggr).
\]
On the other hand,
\begin{align*}
\int_{0}^{\epsilon}x^2 d\varPhi(x)
&=\frac{1}{\sqrt{2\pi}}\int_{0}^{\epsilon}
\biggl(x^2-\frac{x^4}{2}+\frac{x^6}{2^{2}2!}
-\frac{x^8}{2^{3}3!}+\cdots \biggr)dx\\[3pt]
&=\frac{1}{\sqrt{2\pi}} \biggl(\frac{\epsilon^3}{3}
-\frac{\epsilon^5}{10}+\frac{\epsilon^7}{7\cdot2^{2}2!}
-\frac{\epsilon^9}{9\cdot2^{3}3!}+\cdots \biggr)\\[3pt]
&\le\frac{1}{\sqrt{2\pi}}\biggl(\frac{\epsilon^3}{3}
-\frac{\epsilon^5}{10}+\frac{\epsilon^7}{56} \biggr)\\[3pt]
&\le\frac{1}{\sqrt{2\pi}} \biggl(\frac{\epsilon^3}{3}
-\frac{\epsilon^5}{10}+\frac{\epsilon^5}{56} \biggr)
=\frac{1}{\sqrt{2\pi}}\epsilon ^3 \biggl(\frac{1}{3}
-\frac{23}{280}\epsilon^2 \biggr)\\[3pt]
&=\frac{1}{\sqrt{2\pi}}\frac{\epsilon^3}{3}
\biggl(1-\frac{69}{280}\epsilon^2 \biggr)
\le\frac{1}{\sqrt{2\pi}}\frac{\epsilon^3}{3}
\biggl(1-\frac{\epsilon^2}{6} \biggr)\le\frac{\epsilon^2}{3}
\biggl(\varPhi(\epsilon)-\frac{1}{2} \biggr),
\end{align*}
and we immediately get that
\[
\int
_{0}^{\epsilon}x^2 d\varPhi(x)\le
\frac{\epsilon
^2}{3} \biggl(\varPhi (\epsilon)-\frac{1}{2} \biggr).
\]

It is obvious that the density function is bounded.
Moreover, $F(x)$ is symmetric. Therefore, $\mu_1=0$ and $\mu_3=0$.
Furthermore, taking into account (\ref{exp}), consider
\begin{align*}
\si^2&=\int
_{-\infty}^{\infty}x^2dF(x)=
\int
_{|x|\ge
\epsilon}x^2d\varPhi(x)+\int
_{-\theta\epsilon}^{\theta
\epsilon
}x^2
\frac{\varPhi(\epsilon)-\frac{1}{2}}{\theta\epsilon}dx\\[3pt]
&=\int
_{|x|\ge\epsilon}x^2d\varPhi(x)+\frac{\varPhi(\epsilon
)-\frac
{1}{2}}{\theta\epsilon}
\frac{2}{3} (\theta\epsilon )^3 =\int
_{-\infty}^{\infty}x^2d
\varPhi(x)=1.
\end{align*}
This means that $\mu_2=0$.
Consider further the pseudomoments
\begin{align*}
\nu_4&=\int
_{-\infty}^{\infty}x^4
\bigl\llvert d \bigl(F(x)-\varPhi (x) \bigr)\bigr\rrvert =\int
_{-\epsilon}^{\epsilon}x^4
\bigl\llvert d \bigl(F(x)-\varPhi (x) \bigr)\bigr\rrvert
\\[4pt]
&\le\epsilon^4\int
_{-\epsilon}^{\epsilon}\bigl\llvert d
\bigl(F(x)-\varPhi (x) \bigr)\bigr\rrvert \le\epsilon^{4}4 \biggl(
\varPhi(\epsilon)-\frac
{1}{2} \biggr), \\[5pt]
\nu_n^{(1)}(3)&= \int
_{|x|\le\si\sqrt
{n}}x^4\bigl\llvert d
\bigl(F(x)-\varPhi(x) \bigr)\bigr\rrvert =\int
_{-\epsilon
}^{\epsilon}x^4
\bigl\llvert d \bigl(F(x)-\varPhi(x) \bigr)\bigr\rrvert ,
\end{align*}
where $\epsilon$ can be chosen so that $\nu_4\le\frac
{1}{2}e^{-\frac
{3}{2}}$ and $\nu_n^{(1)}(3)\le\frac{1}{2}e^{-\frac{3}{2}}$. Then
\begin{equation*}
\nu_n^{(2)}(3)= \int
_{|x|>\si\sqrt
{n}}|x|^3\bigl\llvert
d \bigl(F(x)-\varPhi(x) \bigr)\bigr\rrvert =0.
\end{equation*}
Hence, condition (ii) of Theorem \ref{t1} holds. Therefore, the
function $F(x)$ satisfies the conditions of Theorem \ref{t1'} with $m=3$.



\begin{thebibliography}{12}

\bibitem{Bojaryshcheva}
\begin{barticle}
\bauthor{\bsnm{Bojaryscheva}, \binits{T.V.}}:
\batitle{Investigation of the rate of convergence of the sums of independent
 random variables }.
\bjtitle{Res. Bull. Uzhhorod Univ., Math. Inform. Ser.}
\bvolume{25}(\bissue{1}),
\bfpage{21}--\blpage{27}
(\byear{2014})
(\bcomment{in {U}krainian})
\end{barticle}
%
\OrigBibText
\begin{barticle}
\bauthor{\bsnm{Bojaryscheva}, \binits{T.V.}}:
\batitle{Investigation of the rate of convergence of the sums of independent
 random variables (in {U}krainian)}.
\bjtitle{Research Bulletin of Uzhhorod University. Mathematics and Informatics
 series}
\bvolume{25(1)},
\bfpage{21}--\blpage{27}
(\byear{2014})
\end{barticle}
\endOrigBibText
\bptok{structpyb}%
\endbibitem

\bibitem{Feller}
\begin{bbook}
\bauthor{\bsnm{Feller}, \binits{W.}}:
\bbtitle{An Introduction to Probability Theory and Its Applications, vol.~II},
\bedition{2}nd edn.
\bsertitle{Wiley Series in Probability and Mathematical Statistics}.
\bpublisher{John Wiley and Sons, Inc.},
\blocation{New York}
(\byear{1971}).
\bid{mr={0270403}}
\end{bbook}
%
\OrigBibText
\begin{bbook}
\bauthor{\bsnm{Feller}, \binits{W.}}:
\bbtitle{An Introduction to Probability Theory and Its Applications. Vol II.
 2nd Ed.}
\bpublisher{Wiley Series in Probability and Mathematical Statistics},
\blocation{New York etc.: John Wiley and Sons, Inc.}
(\byear{1971})
\end{bbook}
\endOrigBibText
\bptok{structpyb}%
\endbibitem

\bibitem{Honak}
\begin{bchapter}
\bauthor{\bsnm{Honak}, \binits{S.V.}},
\bauthor{\bsnm{Sytar}, \binits{I.V.}},
\bauthor{\bsnm{Slyusarchuk}, \binits{P.V.}}:
\bctitle{On one estimate due to {Y}u.{P}.~{S}tudnev}.
In: \bbtitle{Proceedings of the Student Research Conference of the Mathematical Faculty of UzhNu}.
\bsertitle{Series Mathematics and Applied Mathematics},
p.~\bfpage{92}
(\byear{2013})
(\bcomment{in~{U}krainian})
\end{bchapter}%
\OrigBibText
\begin{botherref}
\oauthor{\bsnm{Honak}, \binits{S.V.}},
\oauthor{\bsnm{Sytar}, \binits{I.V.}},
\oauthor{\bsnm{Slyusarchuk}, \binits{P.V.}}:
On one estimate due to {Y}u.{P}.~{S}tudnev (in {U}krainian).
Proceedings of the Student Research Conference of the Mathematical Faculty of
 UzhNu. Series Mathematics and Applied Mathematics,
92
(2013)
\end{botherref}
\endOrigBibText
\bptok{structpyb}%
\endbibitem

\bibitem{Loeve}
\begin{bbook}
\bauthor{\bsnm{Loeve}, \binits{M.}}:
\bbtitle{Probability Theory I}.
\bpublisher{Springer},
\blocation{New York, Heidelberg, Berlin}
(\byear{1977}).
\bid{mr={0651017}}
\end{bbook}
%
\OrigBibText
\begin{bbook}
\bauthor{\bsnm{Loeve}, \binits{M.}}:
\bbtitle{Probability Theory I}.
\bpublisher{Springer},
\blocation{New York, Heidelberg, Berlin}
(\byear{1977})
\end{bbook}
\endOrigBibText
\bptok{structpyb}%
\endbibitem

\bibitem{Paulauskas}
\begin{barticle}
\bauthor{\bsnm{Paulauskas}, \binits{V.M.}}:
\batitle{On the reinforcement of the {L}yapunov theorem}.
\bjtitle{Liet. Mat. Rink.}
\bvolume{9}(\bissue{2}),
\bfpage{323}--\blpage{328}
(\byear{1969}).
\bid{mr={0266280}}
\end{barticle}
%
\OrigBibText
\begin{barticle}
\bauthor{\bsnm{Paulauskas}, \binits{V.M.}}:
\batitle{On the reinforcement of the {L}yapunov theorem}.
\bjtitle{Lietuvos Matematikos Rinkinys}
\bvolume{9(2)},
\bfpage{323}--\blpage{328}
(\byear{1969})
\end{barticle}
\endOrigBibText
\bptok{structpyb}%
\endbibitem

\bibitem{petrov}
\begin{bbook}
\bauthor{\bsnm{Petrov}, \binits{V.V.}}:
\bbtitle{Sums of Independent Random Variables}.
\bpublisher{Springer},
\blocation{New York, Heidelberg, Berlin}
(\byear{1975}).
\bid{mr={0388499}}
\end{bbook}
%
\OrigBibText
\begin{bbook}
\bauthor{\bsnm{Petrov}, \binits{V.V.}}:
\bbtitle{Sums of Independent Random Variables}.
\bpublisher{Springer},
\blocation{New York--Heidelberg--Berlin}
(\byear{1975})
\end{bbook}
\endOrigBibText
\bptok{structpyb}%
\endbibitem

\bibitem{Sl-Pol}
\begin{barticle}
\bauthor{\bsnm{Slyusarchuk}, \binits{P.V.}},
\bauthor{\bsnm{Poljak}, \binits{I.J.}}:
\batitle{Some estimates of the rate of convergence in the central limit theorem}.
\bjtitle{Res. Bull. Uzhhorod Univ., Math. Ser.}
\bvolume{2},
\bfpage{104}--\blpage{107}
(\byear{1997})
(\bcomment{in {U}krainian}).
\bid{mr={2464726}}
\end{barticle}
%
\OrigBibText
\begin{barticle}
\bauthor{\bsnm{Slyusarchuk}, \binits{P.V.}},
\bauthor{\bsnm{Poljak}, \binits{I.J.}}:
\batitle{Some estimates of the rate of convergence in the central limit theorem
 (in {U}krainian)}.
\bjtitle{Research Bulletin of Uzhhorod University. Mathematics series.}
\bvolume{2},
\bfpage{104}--\blpage{107}
(\byear{1997})
\end{barticle}
\endOrigBibText
\bptok{structpyb}%
\endbibitem

\bibitem{Statulyavichus}
\begin{barticle}
\bauthor{\bsnm{Statulevi\v cius}, \binits{V.A.}}:
\batitle{Limit theorems for densities and asymptotic expansions for the
 distributions of sums of independent random variables}.
\bjtitle{Theory Probab. Appl.}
\bvolume{10}(\bissue{4}),
\bfpage{582}--\blpage{595}
(\byear{1965}).
\bid{mr={0193660}}
\end{barticle}
%
\OrigBibText
\begin{barticle}
\bauthor{\bsnm{Statulevi\v cius}, \binits{V.A.}}:
\batitle{Limit theorems for densities and asymptotic expansions for the
 distributions of sums of independent random variables}.
\bjtitle{On the closeness of the distributions of two sums of independent
 random variables.}
\bvolume{10(4)},
\bfpage{582}--\blpage{595}
(\byear{1965})
\end{barticle}
\endOrigBibText
\bptok{structpyb}%
\endbibitem

\bibitem{Studnyev}
\begin{barticle}
\bauthor{\bsnm{Studnyev}, \binits{Y.P.}}:
\batitle{One form of estimating the rate of convergence to a normal law}.
\bjtitle{Ukr. Math. J.}
\bvolume{20},
\bfpage{256}--\blpage{259}
(\byear{1968})
\end{barticle}
%
\OrigBibText
\begin{barticle}
\bauthor{\bsnm{Studnyev}, \binits{Y.P.}}:
\batitle{One form of estimating the rate of convergence to a normal law}.
\bjtitle{Ukrainian Math. J.}
\bvolume{20},
\bfpage{256}--\blpage{259}
(\byear{1968})
\end{barticle}
\endOrigBibText
\bptok{structpyb}%
\endbibitem

\bibitem{Zolotarev1}
\begin{barticle}
\bauthor{\bsnm{Zolotarev}, \binits{V.M.}}:
\batitle{On the closeness of the distributions of two sums of independent
 random variables.}
\bjtitle{Theory Probab. Appl.}
\bvolume{10},
\bfpage{472}--\blpage{479}
(\byear{1965}).
\bid{mr={0189109}}
\end{barticle}
%
\OrigBibText
\begin{barticle}
\bauthor{\bsnm{Zolotarev}, \binits{V.M.}}:
\batitle{On the closeness of the distributions of two sums of independent
 random variables.}
\bjtitle{Theory of Probability and its Applications}
\bvolume{10},
\bfpage{472}--\blpage{479}
(\byear{1965})
\end{barticle}
\endOrigBibText
\bptok{structpyb}%
\endbibitem

\bibitem{Zolotarev2}
\begin{bchapter}
\bauthor{\bsnm{Zolotarev}, \binits{V.M.}}:
\bctitle{Exactness of an approximation in the central limit theorem}.
In: \bbtitle{Proc. 2nd Japan--USSR Sympos. Probab. Theory, Kyoto 1972}.
\bsertitle{Lect. Notes Math.}, vol. \bseriesno{330},
pp. \bfpage{531}--\blpage{543}
(\byear{1973}).
\bid{mr={0443048}}
\end{bchapter}
%
\OrigBibText
\begin{botherref}
\oauthor{\bsnm{Zolotarev}, \binits{V.M.}}:
Exactness of an approximation in the central limit theorem.
Proc. 2nd Japan--USSR Sympos. Probab. Theory, Kyoto 1972, Lect. Notes Math.
 \textbf{330},
531--543
(1973)
\end{botherref}
\endOrigBibText
\bptok{structpyb}%
\endbibitem

\bibitem{Zolotarev}
\begin{bbook}
\bauthor{\bsnm{Zolotarev}, \binits{V.M.}}:
\bbtitle{Modern Theory of Summing the Independent Random Variables}.
\bpublisher{Nauka},
\blocation{Moscow}
(\byear{1986}).
\bid{mr={0917274}}
\end{bbook}
%
\OrigBibText
\begin{bbook}
\bauthor{\bsnm{Zolotarev}, \binits{V.M.}}:
\bbtitle{Modern Theory of Summing the Independent Random Variables}.
\bpublisher{Nauka},
\blocation{Moscow}
(\byear{1986})
\end{bbook}
\endOrigBibText
\bptok{structpyb}%
\endbibitem

\end{thebibliography}
\end{document}